\documentclass[11pt]{amsart}
\newtheorem{theorem}{Theorem}[section]
\newtheorem{lemma}{Lemma}[section]
\newtheorem{proposition}{Proposition}[section]

\theoremstyle{definition}

\theoremstyle{remark}
\newtheorem{remark}{Remark}[section]
\numberwithin{equation}{section}
\newtheorem{corollary}{Corollary}[section]

\numberwithin{equation}{section}

\newcommand{\M}{\sharp_{e}S^n\times S^n\sharp_{c} S^1\times S^{2n-1}}
\newcommand{\N}{\sharp_{c} S^1\times S^{n}}

\newcommand{\Z}{\mathbb Z}
\newcommand{\R}{\mathbb R}

\begin{document}

 \title[Critical points]{Examples of smooth maps with finitely
 many critical points in dimensions $(4,3)$, $(8,5)$ and $(16,9)$}
\author[L.Funar]{Louis Funar}
\address{Institut Fourier BP 74, UMR 5582, Universit\'e de
Grenoble I, 38402 Saint-Martin-d'H\`eres cedex, France}
\email{funar@fourier.ujf-grenoble.fr}

\author[C.Pintea]{Cornel Pintea}
\address{Department of Geometry, "Babe\c{s}-Bolyai" University, 400084 M. Kog\u{a}lniceanu 1,
Cluj-Napoca, Romania}
\email{cpintea@math.ubbcluj.ro}
\author[P.Zhang]{Ping Zhang}
\address{Department of Mathematics, Eastern Mediterranean University, Gazima\~{g}usa, North Cyprus,
via Mersin 10, Turkey}
\email{ping.zhang@emu.edu.tr}

\date{July 19, 2008}

\begin{abstract} We consider manifolds $M^{2n}$  which admit smooth maps 
into a connected sum of $S^1\times S^n$ with  only finitely many critical points, 
for $n\in\{2,4,8\}$, and compute the minimal number of critical points.
\end{abstract}



\maketitle

\section{Introduction and statement of the main result}

Let $\varphi(M^m,N^n)$ denote the minimal number of critical points
of smooth  
maps between the manifolds $M^m$ and $N^n$. 
When superscripts
are specified they denote the dimension of the respective manifolds.
We are interested below in the case when $m\geq n\geq 2$ and 
the manifolds are compact. 
The main problem concerning $\varphi$ is to characterize those
pairs of manifolds for which it is finite non-zero
and then to compute its value (see \cite{AndFun1}).

In \cite{AndFun1} the authors found that,
in small codimension $0\leq m-n-1\leq 3$, if $\varphi(M^m,N^{n+1})$ is finite
then $\varphi(M^m,N^{n+1})\in\{0,1\}$, except for the
exceptional pairs of dimensions
$(m,n+1)\in\{(2,2), (4,3), (4,2), (5,2), (6,3), (8,5)\}$. 
Notice that $(5,3)$ was  inadvertently included 
in \cite{AndFun1} among the 
exceptional pairs, but the proof carries out over this case. 
Moreover, under the finiteness hypothesis,
$\varphi(M,N)=1$ if and only if $M$ is the connected sum 
of a  smooth fibration over $N$ with an exotic sphere and not a fibration 
itself. There are two  essential ingredients in this result. First, there 
are local obstructions to the existence of isolated 
singularities, namely  the germs of smooth maps $\R^m\to \R^n$ 
having an isolated singularity at origin are actually locally topologically 
equivalent to a projection. Thus, these maps are topological fibrations. 
Second, singular points located in a disk cluster together.

The simplest exceptional case is that of (pairs of) surfaces,
which is completely understood by elementary means 
(see \cite{AF2} for explicit computations).
Very little is known for the other exceptional and generic (i.e. $m-n-1\geq 4$)
cases and even the case of pairs of spheres is unsettled yet.
In particular, it is not known whether $\varphi$ is bounded in terms 
only of the dimensions, in general. 

The aim of this note is to find  non-trivial examples 
in dimensions $(4,3)$, $(8,5)$  and $(16,9)$ inspired by the early 
work of Antonelli (\cite{Ant1,Ant3}). 
The smooth maps considered in \cite{Ant3} are so-called 
Montgomery-Samelson fibrations with finitely many singularities where
several fibers are pinched to points. According to \cite{Tim} 
these maps should be locally topologically equivalent to a cone 
over the Hopf fibration, in a neighborhood of a critical point. 

The main ingredient of our approach is the existence of global obstructions 
of topological nature to the clustering of genuine critical points in these 
dimensions. This situation seems rather exceptional and it permits us 
to obtain the precise value of $\varphi$ using only basic 
algebraic topology. 

Our computations show that $\varphi$ 
can take arbitrarily large even values. Thus the
behavior of $\varphi$ is qualitatively different from
what it was seen before in \cite{AndFun1}.

\begin{theorem}\label{comput}
Let $n\in \{2,4,8\}$, $e\geq c\geq 0$, with $c\neq 1$, and  
$\Sigma^{2n}$ be a homotopy $2n$-sphere.  If $n=2$ assume further that 
$\Sigma^{4}\setminus {\rm int}(D^4)$ embeds  smoothly into $S^4$, where 
$D^4$ is a smooth  4-disk.  
Then 
\[\varphi(\Sigma^{2n}\M,\N)=2e-2c+2\]
Here  $\N=S^{n+1}$ if $c=0$ and $\M=S^{2n}$
if $e=c=0$. 
\end{theorem}

The structure of the proof of the theorem is as follows.
We prove Proposition \ref{lowbound} which 
yields  a lower bound for the number of critical values
derived from topological obstructions of algebraic nature. 
The existence of a  non-trivial lower bound is not obvious 
since one might think that several singularities
could combine into a single more complicated singularity.
However, the proof uses only standard techniques of algebraic 
topology. 
The next step taken in section \ref{fibersum} is to construct
explicit smooth maps with any even 
number of singularities. This follows by taking fiber sums of elementary
blocks of maps coming naturally from Hopf fibrations. 
This construction is an immediate generalization of  the one 
considered by Antonelli in the case of two elementary blocks in (\cite{Ant1}, p.185-186). 
Then Proposition \ref{fsum} concludes the proof.

\begin{remark}
Observe that  $S^1\times S^{2n-1}$ fibers over
$S^1\times S^n$, when $n\in\{2,4,8\}$ so that the formula from Theorem \ref{comput} is still valid for 
$\Sigma^{2n}=S^{2n}$, $e=0$ and $c=1$. However,
we do not know  how to  evaluate $\varphi$ when $e\leq c-1$.
The present methods do not work for $e\geq c=1$ either. 
\end{remark}


\vspace{0.2cm}

{\small 
{\bf Acknowledgements}. The authors are indebted to  
Dennis Sullivan, Yuli Rudyak and Andr\'{a}s Sz\H{u}cs for useful discussions 
on this topic, during the M. M. Postnikov Memorial Conference at Bedlewo, 
June 2007. L. Funar was partially supported by the 
ANR Repsurf: ANR-06-BLAN-0311. 
C. Pintea was partially supported by the CNCSIS grant of type 
A, 8/1467/2007 and partially by the CNCSIS grant PN II, ID 523/2007.
P. Zhang thanks Jianzhong Pan, Haibao Duan and Kaiming Zhao of Institute of 
Mathematics of CAS in Beijing for their hospitality when part of this note 
was written in the summer of 2007. 
}

\section{A lower bound for the number of critical values}\label{lowbd}

\begin{proposition}\label{lowbound}
For any dimension $n\geq 2$, homotopy $2n$-sphere $\Sigma^{2n}$ 
and non-negative integers $e$ and $c$, 
with $c\neq 1$ we have: 
\[\varphi(\Sigma^{2n}\M,\N)\geq 2e-2c+2\]
Here  $\N=S^{n+1}$ if $c=0$ and $\M=S^{2n}$
if $e=c=0$. 
\end{proposition}

We will prove, more generally, the following:
\begin{proposition}\label{genca}
Let $M^{2n}$ and $N^{n+1}$ be  closed  connected orientable 
manifolds and $n\geq 2$. Assume that
$\pi_1(M)\cong \pi_1(N)$ is a free group ${\mathbb F}(c)$ on $c$ generators, 
$c\neq 1$  (with ${\mathbb F}(0)=0$) , 
$\pi_j(M)=\pi_j(N)=0$, for $2\leq j \leq n-1$ and $H_{n-1}(M)=0$. 
Then $\varphi (M,N)\geq \beta_{n}(M)-2c+2$, where $\beta_k$ denotes the
$k$-th Betti number.
\end{proposition}
\begin{proof}
Let $B=B(f)$ denote the set of critical values of a smooth map
$f:M\to N$. We will prove that the
cardinality $|B|$ of $B(f)$ satisfies 
$|B|\geq \beta_{n}(M)-2c+2$, which will imply our claim. 
Set $V=f^{-1}(B(f))\subset M$. We can assume that $f$ has finitely many 
critical points, since otherwise the claim of Proposition 
\ref{genca} would be obviously verified. 

The following two Lemmas do not depend on the homotopy assumptions 
of Proposition \ref{genca}. 

\begin{lemma}\label{finite}
If $A$ is a nonempty finite subset of a connected closed orientable 
manifold $N^{n+1}$, then $\beta_{n}(N\setminus
A)=\beta_{n}(N)+|A|-1$.
\end{lemma}
\begin{proof}Clear from the homology exact sequence of the pair $(N,N\setminus A)$. 
\end{proof}

\begin{lemma}\label{connect}
 If $M^{n+q+1}$ and $N^{n+1}$ are smooth manifolds and
$f:M\rightarrow N$ is a smooth map with finitely many critical
points, then the inclusions $M\setminus
V\hookrightarrow M$ and $N\setminus B\hookrightarrow N$ are $n$-connected.
\end{lemma}
\begin{proof}
This is obvious for  $N\setminus
B\hookrightarrow N$. It remains to prove
that $\pi_k(M,M\setminus V)\cong
0$ for $k\leq n$. Take $\alpha:(D^{k},
S^{k-1})\rightarrow (M, M\setminus V)$  to be an arbitrary
smooth map of pairs.
Since the critical set $C(f)$ of $f$ is finite and contained in
$V$, there exists a small homotopy of $\alpha$ relative to
the boundary such that the image $\alpha(D^k)$ avoids $C(f)$.
By compactness there exists a neighborhood $U$ of $C(f)$ 
consisting of disjoint balls centered at the critical points such that 
$\alpha(D^k)\subset M\setminus U$.  We can arrange by a small isotopy 
that $V$ becomes transversal to $\partial U$.

Observe further that $V\setminus U$ consists of regular
points of $f$ and thus it is a properly embedded sub-manifold of
$M\setminus U$. General transversality arguments show that $\alpha$
can be made transverse to  $V\setminus U$ by a small
homotopy. By dimension counting this means
that $\alpha(D^k)\subset M\setminus U$ is disjoint from $V$ and thus
the class of $\alpha$ in  $\pi_k(M,M\setminus V)$ vanishes.
\end{proof}

The restriction of $f$ to $M\setminus V$ is a proper submersion 
and thus  the restriction $f|_{M\setminus V}$ is an open map.  
In particular, $f(M\setminus V)\subset N\setminus B$ is an open subset. 
On the other hand, the closed map lemma  states that a proper map between 
locally compact Hausdorff spaces is also closed.  Thus $f(M\setminus V)$ is 
also closed in $N\setminus B$ and hence $f(M\setminus V)=N\setminus B$. 
According to Ehresmann's theorem, the restriction 
$f|_{M\setminus V}$ is then a locally trivial smooth fibration over $N\setminus B$ with
compact smooth fiber $F^{n-1}$ (see \cite{Dim}).

\begin{lemma}
Assume that $c\neq 1$. Then  the
generic fiber $F$ is homotopy equivalent to the $(n-1)$-sphere. 
\end{lemma}
\begin{proof}
When $c=0$ the claim is a simple consequence of the homotopy sequence 
of the fibration  $M\setminus V\to N\setminus B$. 
 
Let us assume henceforth that $c\geq 2$. 
Consider the last terms of the homotopy exact sequence of this fibration:
\[ \to \pi_1(M\setminus V)\stackrel{f_*}{\to} \pi_1(N\setminus B)\stackrel{p}{\to} \pi_0(F)\to \pi_0(M\setminus V)\to \pi_0(N\setminus B)\]
From Lemma \ref{connect} $M\setminus V$ and $N\setminus B$ are connected and
$\pi_1(M\setminus V)\cong \pi_1(N\setminus B)\cong {\mathbb F}(c)$. 
If $F$ has $d \geq 2$ connected components then the kernel $\ker p$ of $p$ is a
finite index proper subgroup of  the free non-abelian
group ${\mathbb F}(c)$. Thus, by the Nielsen-Schreier theorem,  
$\ker p$ is a free group of rank $d(c-1)+1$, where 
$d$ is the number of components of $F$, and hence of rank larger than $c$.
On the other hand, by exactness of the sequence above, $\ker p$ is also the image of $f_*$ and
thus it is a group of rank at most $c$. This contradiction shows that
$F$ is connected.

If $n=2$ then $F$ is a circle, as claimed.

Let now $n>2$. We obtained above that
$f_*$ is surjective. Since finitely generated free groups are
Hopfian any surjective homomorphism ${\mathbb F}(c)\to {\mathbb F}(c)$
is also injective. Since $\pi_2(N\setminus B)\cong \pi_2(N)=0$
and $f_*$ is injective we derive that $\pi_1(F)=0$.
The remaining terms of the homotopy exact sequence of the fibration
and Lemma \ref{connect} show then that
$\pi_j(F)=0$ for $2\leq j\leq n-2$. Thus $F$ is a homotopy sphere. 
\end{proof}

\begin{lemma}\label{surj}
Suppose that $B\neq\emptyset$. 
\begin{enumerate}
\item We have
$H_1(N\setminus B)\cong \Z^c$, $H_{n}(N\setminus B)=\Z^{|B|+c-1}$ and $H_{n+1}(N\setminus B)=0$.
\item If $n > 2$ then $H_{n-1}(M\setminus V)=0$.
\item The homomorphism $H_n(M\setminus V)\to H_n(M)$ induced by the
inclusion map is surjective.
\end{enumerate}
\end{lemma}
\begin{proof}
The first two assertions are consequences of Lemma \ref{finite},
Lemma \ref{connect} and standard algebraic topology.
For instance, $H_1(N\setminus B)\cong H_1(N)=\Z^c$.
The last claim follows from Lemma \ref{connect} and the long exact 
sequence in homology of the pair $(M,M\setminus V)$.
\end{proof}

\begin{lemma}\label{rank}
If $B\neq \emptyset$ and $c\neq 1$ then the rank of $H_{n}(M\setminus V)$ is $2c+|B|-2$.
\end{lemma}
\begin{proof}
The Gysin sequence of the fibration  $M\setminus V\to N\setminus B$ (whose fiber is  
a homotopy sphere)  reads:
\[ \to H_m(M\setminus V)\to  H_m(N\setminus B)\to H_{m-n}(N\setminus B)
\to H_{m-1}(M\setminus V)\to \]
Consider the exact subsequence
\[ H_{n+1}(N\setminus B)\to H_1(N\setminus B)\to H_{n}(M\setminus V)
\to H_n(N\setminus B)\to H_0(N\setminus B)\to H_{n-1}(M\setminus V)\]
If $n>2$ then  the first and the last terms vanish.

The Euler characteristic of this subsequence is zero by exactness and thus
the rank of $H_{n}(M\setminus V)$ is $2c+|B|-2$ by Lemma \ref{surj}.

When $n=2$, we can complete the exact sequence above by
adding one more term to its right, namely
$H_{1}(M\setminus V)\stackrel{f_*}{\to} H_{1}(N\setminus B)$.
However, $f_*$ is actually the map induced in homology by the isomorphism
$f_*:\pi_1(M)\to \pi_1(N)$ and thus an isomorphism itself.
The argument with the Euler characteristic can be applied again and
yields the claimed result.
\end{proof}

From Lemma \ref{rank} and Lemma \ref{surj} (3) we derive that
\[ 2c+|B|-2 \geq \beta_{n}(M)\]
and the proposition is proved.
\end{proof}

\begin{corollary}
If $M^{2n}$ is a smooth $(n-1)$-connected
closed manifold, then
$$\varphi(M,\Sigma^{n+1})\geq \beta_{n}(M)+2,$$
where $\Sigma^{n+1}$ is a homotopy sphere.
\end{corollary}

\begin{remark}
The present approach does not work for $c=1$. In fact, fibers 
might have several connected components, each one being  
a homotopy sphere.  In the absence of an upper bound of the number  of components 
the Leray-Serre spectral sequence leads only to a trivial lower bound for the 
number of critical values. 
\end{remark}

\section{Fiber sums of suspensions of Hopf fibrations}\label{fibersum}
\begin{proposition}\label{fsum}
Let $n\in \{2,4,8\}$, $e\geq c\geq 0$, with $c\neq 1$, and  
$\Sigma^{2n}$ be a homotopy $2n$-sphere.  If $n=2$ assume further that 
$\Sigma^{4}\setminus {\rm int}(D^4)$ embeds  smoothly into $S^4$, where 
$D^4$ is a smooth  4-disk.  
Then 
\[\varphi(\Sigma^{2n}\M,\N)\leq 2e-2c+2\]
\end{proposition}
\begin{proof}
Recall from \cite{AndFun1} that
$\varphi(S^{2n},S^{n+1})=2$ if $n=2,4$ or $8$. This is realized
by taking suspensions of both spaces in the Hopf fibration
$h:S^{2n-1}\rightarrow S^n$, where $n=2,4$ or $8$, and then
smoothing the new map at both ends. The extension
$H:S^{2n}\rightarrow S^{n+1}$ has precisely two critical
points.  This is also the basic example of a Montgomery-Samelson
fibration with finitely many singularities, as considered in
\cite{Ant3}. Antonelli has considered in \cite{Ant1} manifolds 
which admit maps with two critical points into spheres, by gluing together 
two copies of $H$. 

Our aim  is to define fiber sums of Hopf fibrations
leading to other examples of pairs of manifolds with finite $\varphi$ 
using Antonelli's construction for more general  gluing patterns. 
Identify $S^{n+1}$ (and respectively $S^{2n}$) with the suspension
of $S^n$ (respectively $S^{2n-1}$) and thus equip it with the coordinates
$(x,t)$, where $|x|^2+t^2=1$, and $t\in [-1,1]$.  We call the coordinate $t$
the height of the respective point.
The suspension $H$ is then given by:
\[H(x,t)=\left(\psi(|x|)h\left(\frac{x}{|x|}\right), t\right)\]
where $\psi:[0,1]\to [0,1]$ is a smooth  increasing function infinitely flat at $0$ such that 
$\psi(0)=0$ and $\psi(1)=1$.

Pick up a number of points $x_1,x_2,\ldots, x_k\in S^{n+1}$ and
their small enough disk neighborhoods $x_i\in D_i\subset S^{n+1}$,
such that:
\begin{enumerate}
\item the projections of $D_i$ on the height coordinate axis are disjoint;
\item the $D_i$'s do not contain the two poles, i.e. their projections on the
height axis are contained in the open interval $(-1,1)$.
\end{enumerate}
Let $A_k$ be the manifold with boundary obtained by deleting  from
$S^{n+1}$ of the interiors of the disks $D_i$, for $1\leq i\leq k$.
Let also $B_k$ denote the preimage $H^{-1}(A_k)\subset S^{2n}$ by the
suspended Hopf map.  Since $H$  restricts to a trivial fibration
over the disks $D_i$ it follows that  $B_k$ is a manifold, each
one of its boundary components being diffeomorphic to $S^{n-1}\times S^n$.
Moreover, the boundary components are endowed with a natural
trivialization induced from $D_i$.

Let now $\Gamma$ be a finite connected graph. 
To each vertex $v$ of valence $k$ we associate
a block $(B_v,A_v, H|_{B_v})$, which will be denoted  $(B_k,A_k, H|_{B_k})$, 
when we want to emphasize the dependence on the number of boundary components.
Each boundary component of $A_v$ or $B_v$ corresponds to an edge incident
to the vertex $v$.
We define the fiber sum along $\Gamma$
as the following  triple $(B_{\Gamma}, A_{\Gamma}, H_{\Gamma})$:
\begin{enumerate}
\item $A_{\Gamma}$ is the result of gluing the manifolds with boundary
$A_v$, associated to the vertices $v$ of $\Gamma$, by identifying, for each
edge $e$ joining the vertices $v$ and $w$ (which might coincide) 
the pair of boundary components in $A_v$ and $A_w$ 
corresponding to the edge $e$. The identification is made by using  an orientation-reversing 
diffeomorphism of the boundary spheres. 
\item $B_{\Gamma}$ is the result of gluing the manifolds with boundary
$B_v$, associated to the vertices $v$ of $\Gamma$, by identifying, for each
edge $e$ joining the vertices $v$ and $w$ (which might coincide)
the boundary components in $B_v$ and $B_w$ corresponding to the pair 
of boundary components in $A_{\Gamma}$ associated to $e$. Gluings in $B_{\Gamma}$ are realized by 
some orientation-reversing diffeomorphisms which respect the product structure over 
boundaries of  $A_{v}$ and $A_{w}$.  
\item As the boundary components are identified 
the natural trivializations of the boundary components
of $B_v$ agree in pairs. Thus the maps $H_v$ induce a well-defined map
$H_{\Gamma}:B_{\Gamma}\to A_{\Gamma}$.
\end{enumerate}
In the case where the graph $\Gamma$ consists of two vertices joined by an edge 
this construction is essentially that given in (\cite{Ant1}, p.185-186).

\begin{proposition}
The map $H_{\Gamma}:B_{\Gamma}\to A_{\Gamma}$ has $2m$ critical points, where
$m$ is the number of vertices of $\Gamma$.
\end{proposition}
\begin{proof} Clear, by construction.
\end{proof}
We say that $\Gamma$ has $c$ independent cycles if the rank of
$H_1(\Gamma)$ is $c$. This is equivalent to ask $\Gamma$
to become a tree only after removal of at least  $c$ edges.
Moreover,  $c=e-m+1$ where $e$ denotes the number of edges.
\begin{proposition}
If $\Gamma$ has $e$ edges and $c$ cycles, i.e. $e-c+1$ vertices,
then for a suitable choice of the gluing diffeomorphisms data  
$B_{\Gamma}$ is diffeomorphic to $\Sigma^{2n}\M$ (where $\Sigma^{2n}$ is a 
homotopy sphere, which is trivial when $n=2$), while
$A_{\Gamma}$ is diffeomorphic to
$\N$. Here  $\N$ states for $S^{n+1}$ when $c=0$.

\end{proposition}
\begin{proof}
The sub-blocks $A_k$ are diffeomorphic to the connected sum of $k$ copies 
of disks $D^{n+1}$
out of their boundaries.  When gluing  together two such distinct sub-blocks 
(since there is an edge in $\Gamma$ joining the corresponding vertices) 
the respective pair of disks leads to a factor $D^{n+1}\cup_{\mu} D^{n+1}$, where 
$\mu:S^n\to S^n$ is the identification map.  If  $\mu$ is a reflection  
then the factor  $D^{n+1}\cup_{\mu} D^{n+1}$ is the double of $D^{n+1}$ and 
hence  diffeomorphic to $S^{n+1}$. 
 
When gluing all sub-blocks in the pattern of the graph $\Gamma$
the only non-trivial contribution comes from the cycles. Each cycle
of $\Gamma$ introduces a 1-handle. Thus the manifold $A_{\Gamma}$
is  diffeomorphic to $\sharp_{c} S^1\times S^{n}$.

Further  we have a similar result for the sub-blocks $B_k$: 
\begin{lemma}
The  sub-blocks $B_k$ are diffeomorphic to the  connected sum of 
$k$ copies of the product $S^{n}\times D^{n}$ out of their boundaries. 
\end{lemma}
\begin{proof} 
One obtains $B_k$ by deleting out $k$ copies of $H^{-1}(D_i)$; each  $H^{-1}(D_i)$ is
a tubular neighborhood of the (generic) fiber of $H$ and thus
diffeomorphic to $S^{n-1}\times D^{n+1}$.

When $k=1$ the generic fiber of $H$ is an $S^{n-1}$ embedded in $S^{2n}$, namely 
the  image  of the fiber of the Hopf fibration  in the 
suspension sphere $S^{2n}$.  The generic fiber is unknotted in $S^{2n}$,  as an immediate 
consequence of Haefliger's classification of smooth embeddings. In fact, according to \cite{Hae3}, any
smooth embedding of $S^{k}$ in $S^{m}$ is unknotted, 
i.e. isotopic to the boundary of a standard ball, if the dimensions satisfy  the meta-stable range 
condition $k<\frac{2}{3}m-1$.
This implies that the complement of a regular neighborhood of the fiber is
diffeomorphic to the complement of a standard sphere and thus to
$S^{n}\times D^{n}$.

When $k\geq 2$ we remark that the fibers over the points $x_i\in D_i$ lie at
different heights and thus they are contained in disjoint slice 
spheres of the suspension $S^{2n}$. This implies that these fibers
are unlinked, i.e. isotopic to the boundary of a set of disjoint standard
balls. Thus the complement of  a regular neighborhood of
their union is diffeomorphic to the connected sum of their
individual complements, and therefore to  the  connected sum of 
$k$ copies of the product
$S^{n}\times D^{n}$ out of their boundaries.
\end{proof}

Let us stick for the moment to the case when $k=1$ and we have two diffeomorphic sub-blocks 
$B_v$ and $B_w$, each one having one boundary  component, to be glued together. 
We choose the identification diffeomorphism $\nu:\partial B_v\to \partial B_w$ 
to be the one  from the construction of the double of $B_v$.  Observe that 
the maps $B_v\to A_v$ and $B_w\to A_w$ glue together to form a well-defined smooth map 
$B_v\cup_{\nu}B_w\to A_v\cup_{\mu}A_w$, as already noticed in (\cite{Ant1}, p.185). 

\begin{lemma}
The factor $B_v\cup_{\nu} B_w$ is diffeomorphic to $\Sigma^{2n}\sharp S^{n}\times S^{n}$, where 
$\Sigma^4=S^4$. 
\end{lemma}
\begin{proof}
Consider first the case $n=2$, which is the most interesting  one since the result cannot follow from 
general classification results.  The sub-block  $D^2\times S^2$ can be easily 
described by a Kirby diagram (see \cite{Go}, chapter 4), which encodes its handlebody structure. 
As $D^2\times S^2$ is  obtained from $D^4$ by throwing away the regular 
neighborhood of an unknotted circle (i.e. a $1$-handle) it can be described as the result of attaching 
the dual 2-handle on an unknotted circle with framing $0$. There is also a dual 
handlebody decomposition of $D^2\times S^2$ in which each $j$-handle generates a $4-j$ handle. 
The double of $D^2\times S^2$ is then described by putting together the two 
handlebody descriptions (the usual one and the dual one) and thus is made of $D^4$ with 
two 2-handles and finally a 4-handle capping off the boundary component.  

Attaching maps of $4$-handles are  orientation preserving diffeomorphisms of $S^3$, and by a classical 
result of Cerf these are isotopic to identity. Thus there exists a unique way to attach a 4-handle 
to a 4-manifold with boundary $S^3$. By the way, recall that a theorem of Laudenbach and 
Poenaru  (\cite{LP}) shows  that there is only one way up to global diffeomorphism 
to attach $3$-handles and $4$-handles to a $4$-manifold with boundary $\sharp_k S^1\times S^2$ 
in order to obtain a closed manifold.   

Now it is easy to see that the new 2-handle  (in the handlebody structure of the double 
of $D^2\times S^2$) is attached along a meridian  circle of the former 2-handle with 
$0$ framing. Thus a Kirby diagram of the double of $D^2\times S^2$ consists of a Hopf link 
with both components having framing $0$, and it is well-known that this diagram is 
also that of $S^2\times S^2$.  See also (\cite{Go}, Example 4.6.3)  for more details.

This argument applies as well for  $n\geq 3$. We have a handle decomposition of 
$D^n\times S^n$ as $D^{2n}$ with one $n$-handle attached. The set of framings on 
a sphere $S^{n-1}$ in $\partial D^{2n}$ is acted  upon freely transitively by 
$\pi_{n-1}(O(n))$. Moreover $\pi_3(O(4))\cong \pi_7(O(8))\cong\Z\oplus \Z$ (see \cite{Mi}).  
Then the $n$-handle is attached on an unknotted $(n-1)$-sphere with trivial framing,  
i.e. the $(0,0)$-framing.  Observe that this is the canonical framing associated to 
the identity  attaching  map ${\rm id}_{S^{n-1}\times D^n}$ (see e.g. \cite{Go} Example 4.1.4.(d)).  
Further the double of $D^n\times S^n$  is obtained by putting together 
the usual handlebody and its dual. As above 
we can describe the double as the result of attaching two $n$-handles and one 
$2n$-handle. The dual $n$-handle is attached on a meridian $(n-1)$-sphere which links 
once the  former  attaching $(n-1)$-sphere and has trivial framing.  
The union of the two spheres is the analogue of the Hopf link in $S^{2n-1}=\partial D^{2n}$.  
As it is well-known $S^n\times S^n$ can also be obtained by adding two $n$-handles 
along this high-dimensional  trivially-framed Hopf link and a $2n$-handle. 

The only difference between the cases $n>2$ and $n=2$  
is that the result of attaching a $2n$-handle for $n>2$
is not unique, as there  might exist diffeomorphisms of $S^{2n-1}$ which are not isotopic to identity. 
However,  detaching  and   then reattaching  a $2n$-handle with a reflection diffeomorphism 
as gluing map   will create an exotic sphere (for $n\geq 4$)  and thus the 
double is diffeomorphic to $\Sigma^{2n}\sharp S^n\times S^n$ for some homotopy sphere 
$\Sigma^{2n}$. 
 \end{proof}

When gluing all sub-blocks in the pattern of the graph $\Gamma$ such that each identification map 
is  $\nu$ then each pair of sub-blocks determines a factor $\Sigma^{2n}\sharp S^n\times S^n$. 
If there are no cycles in $\Gamma$ then we obtain a connected sum of such factors, namely 
$\Sigma^{2n}\sharp_e S^n\times S^n$.  Finally, 
the only additional non-trivial contribution comes from the cycles. Each cycle
of $\Gamma$ introduces an extra 1-handle. 
Thus the manifold $B_{\Gamma}$ is diffeomorphic to  $\Sigma^{2n}\M$.
\end{proof}

In order to prove Proposition \ref{fsum} it suffices now to show that  one can attach a homotopy sphere 
to the manifolds $B_{\Gamma}$ and still have the same number of critical points. 
This can be realized by removing a small disk centered at
a critical point and gluing it back differently, when $n\neq 2$,  
and respectively gluing back a homotopy 4-disk, when $n=2$. 
We consider only those homotopy 4-disks which embed  smoothly into $S^4$. 
In this way, using the theorem of Huebsch and Morse for $n=2$ 
(see \cite{HM} and  also \cite{AndFun1} where this argument is carried out in detail) we obtain 
a smooth map with the same (non-zero) number of  critical points, namely $2e-2c+2$. 
Since the homotopy spheres form a finite abelian group under the connected sum 
one can obtain this way all manifolds of the form $\Sigma^{2n}\M$. 
\end{proof}
\begin{remark}
Recall that the group  $\Theta^k$ of homotopy $k$-spheres 
is  $\Theta^8=\Theta^{16}=\Z/2\Z$.  
\end{remark}

\begin{remark}
By twisting $\mu$  by a  diffeomorphism of $S^n$ which is not isotopic to identity 
(e.g. when $n=8$) one could obtain exotic spheres factors in $A_{\Gamma}$. 
More interesting  examples correspond to twisting $\nu$ by some  orientation preserving 
diffeomorphism   $\eta:S^{n-1}\times S^n\to S^{n-1}\times S^n$ 
which still respect  the product structure.  
For instance we can consider some  $\eta$ induced from a map
$S^{n-1}\to SO(n+1)$ whose homotopy class is an element of $\pi_{n-1}(SO(n+1))$. 
It seems that all examples obtained by twisting  are still diffeomorphic to $\Sigma^{2n}\M$.  
\end{remark}

\section{Examples with $\varphi=1$}
The result of \cite{AndFun1} shows that
if $\varphi(M^m,N^{n+1})$ is finite non-zero
(small codimension non-exceptional dimensions) then
$\varphi(M^m,N^{n+1})=1$ and $M^m$ should be diffeomorphic
to $\Sigma^m\sharp \widehat{N}$, where $\Sigma^m$ is an exotic sphere and 
$\widehat{N}$ is the total space of a smooth fibration, such that $M^m$ is not fibered over $N$.
Actually this construction might produce 
non-trivial examples in any codimension. 

\begin{proposition}
If $\Sigma^m$ is an exotic sphere  (for $m=4$ we assume that $\Sigma^4\setminus {\rm int}(D^4)$ 
embeds smoothly in $S^4$) and $\widehat{N}\to N$ a  smooth fibration then 
$\varphi(\Sigma^m\sharp \widehat{N}, N)\in\{0,1\}$.
\end{proposition}
\begin{proof}
We obtain $\Sigma^m\sharp \widehat{N}$ from $\widehat{N}$ by excising
a ball $D^{n+1}$ and gluing it (or a homotopy 4-disk when $m=4$) 
back by means of a suitable diffeomorphism
$h$ of its boundary.
By a classical result of  Huebsch and Morse (\cite{HM}), there exists
a smooth homeomorphism  $\Sigma^m\sharp \widehat{N}\to \widehat{N}$
which has only one critical point located in the ball
$D^{n+1}$. This provides a smooth map $\Sigma^m\sharp \widehat{N}\to N$
with one critical point. 
\end{proof}
\begin{remark}
Notice however that $\Sigma^m\sharp \widehat{N}$
might still be fibered over $N$, although not diffeomorphic 
to $\widehat{N}$. This is so when
$\widehat{N}\to N$ is the Hopf fibration $S^7\to S^4$ and 
 $\Sigma^7\sharp \widehat{N}$ is a Milnor exotic sphere, namely 
a $S^3$-fibration over $S^4$ with Euler class $\pm 1$.
\end{remark}

\begin{remark}
The manifold $M^m=\Sigma^m\sharp S^{m-n-1}\times S^{n+1}$ is not diffeomorphic 
to $S^{m-n-1}\times S^{n+1}$ if  $\Sigma^m$ is an exotic sphere (see  
\cite{Sch}). Thus, the proposition above yields effective examples 
where $\varphi=1$.

If $\Sigma^8$ is the exotic 8-sphere 
which generates the group $\Theta^8=\Z/2\Z$ then    
$\varphi(\Sigma^8\sharp S^3\times S^5,S^5)=1$.
In fact $M^8=\Sigma^8\sharp S^3\times S^5$
is homeomorphic but not diffeomorphic to $S^3\times S^5$. 
Assume the contrary, namely that $M^8$ smoothly fibers over $S^5$.
Then the fiber should be a homotopy 3-sphere and hence $S^3$, by the
Poincar\'e Conjecture. The $S^3$-fibrations 
over $S^5$ are classified by  the elements of $\pi_4(SO(4))\cong \Z/2\Z\oplus \Z/2\Z$. 
There exist  precisely two homotopy types among the $S^3$-fibrations over $S^5$  
which admit cross-sections (see \cite{JW1}, p.217).  If $M^8$ is a $S^3$-fibration 
then it should have a cross-section because it is homotopy equivalent to $S^3\times S^5$ and 
the existence of a cross-section is a homotopy invariant 
(see \cite{JW1}, p.196, \cite{JW}, p.164).  However the two homotopy types correspond 
to  two distinct isomorphism types as spheres bundles. In fact they are classified by 
the image of $\pi_4(SO(3))\cong \Z/2\Z$ into $\pi_4(SO(4))$.  
This means that a $S^3$-fibration having a cross-section  
is either homotopy equivalent to the trivial fibration and then it is isomorphic to  the 
trivial fibration or else it has not the same homotopy type as $S^3\times S^5$.  
Observe also that there is only one $O(4)$-equivalence class and thus precisely 
two isomorphism classes of such $S^3$-fibrations without cross-sections  (\cite{JW}, p.164) . 
In particular, non-trivial $S^3$-fibrations  over $S^5$ cannot be homeomorphic to 
$M^8$ and this  contradiction shows that $M^8$ cannot smoothly fiber over $S^5$.  
\end{remark}

\bibliographystyle{amsplain}

\end{document}